\documentclass[a4paper, 10pt]{amsart}

\usepackage[T1]{fontenc}
\usepackage[english]{babel}

\newtheorem{theo}{Theorem}[section]
\newtheorem{prop}[theo]{Proposition}
\newtheorem{defi}[theo]{Definition}
\newtheorem{lemm}[theo]{Lemma}
\newtheorem{coro}[theo]{Corollary}
\newtheorem{conj}[theo]{Conjecture}
\newtheorem{rema}[theo]{Remark}

\newcommand{\mb}{\mathbb}
\newcommand{\mc}{\mathcal}
\newcommand{\mf}{\mathfrak}

\newcommand{\ra}{\rightarrow}

\usepackage{amssymb}
\usepackage{amsmath}

\usepackage[latin1]{inputenc}

\title{Density of crystalline points on unitary Shimura varieties}
\date {6th June 2012}
\author{Przemys\l aw Chojecki}
\email{chojecki@math.jussieu.fr}

\begin{document}

\def\smfbyname{}

\begin{abstract}
We prove that crystalline points are dense in the spectrum of the completed Hecke algebra on unitary Shimura varieties.
\end{abstract}
\maketitle
\tableofcontents

\section{Introduction}

Recently, Matthew Emerton has proved local-global compatibility conjecture for $GL_2 /\mb{Q}$ in \cite{em1}. The proof (of the weak version of compatibility) crucially relies on the density of crystalline points in the completed cohomology of modular curves. We take on this result and generalise it to unitary Shimura varieties considered by Harris-Taylor in \cite{ht}. This is done in corollary 4.12. The result might be seen as the automorphic analogue of the fact, proved recently by Kentaro Nakamura in \cite{na}, that for a p-adic field $K$, $n$-dimensional crystalline representations of $Gal(\bar{K}/K)$ are Zariski dense in the rigid analytic space associated to the universal deformation ring of a $n$-dimensional mod $p$ representation of $Gal(\bar{K}/K)$. Actually the proof of this result does not use in any essential way the fact that we are dealing with unitary Shimura varieties and we could generalize it further to many other Shimura varieties of PEL-type. We refreined from doing so, mostly because of the direct link to known results about Galois representations (which are discussed in section 5).

The techniques which we use in the proof are those of Emerton and we are fairly close to his exposition, though the details differ due to some technical difficulties. For example, we are no longer dealing with curves and we have to be careful about the higher cohomology groups. We circumvent it by introducing a notion of a cohomologically Eisenstein ideal and localising all the cohomology groups at a fixed cohomologically non-Eisenstein ideal. In the last section, we will give some criteria due to Emerton-Gee and Helm for an ideal of a Hecke algebra to be cohomologically non-Eisenstein to show that it is a very natural definition.
\newline
\newline
\textbf{Acknowledgements.} I would like to thank heartfully Jean-Francois Dat, who provided me with several very useful insights and answered many of my questions.

\section{Notation and definitions}

Let $L$ denote an imaginary quadratic field in which $p$ splits. We will let $c$ denote the complex conjugation. Choose a prime $u$ above $p$. Let $F^+$ denote a totally real field of degree $d$. Set $F = L F^+$. We will assume that $p$ is totally decomposed in $F$. Let $D/F$ be a division algebra of dimension $n^2$ such that $F$ is the centre of $D$, the opposite algebra $D ^{op}$ is isomorphic to $D \otimes _{L,c} L$ and $D$ is split at all primes above $u$. We choose an involution of the second kind $*$ on $D$ and assume that there exists a homomorphism $h : \mb{C} \ra D _{\mb{R}}$ for which $b \mapsto h(i)^{-1} b^* h(i)$ is a positive involution on $D _{\mb{R}}$.

We define the reductive group
$$G(R) = \{ (\lambda ,g) \in R^{\times} \times  D^{op} \otimes _{\mb{Q}} R | g g^* = \lambda \}$$
We will assume that $G$ is a unitary group of signature $(n-1,1)$ for one (fixed) infinite place and a unitary group of signature $(0,n)$ at all the other infinite places, so that we are in the situation considered by Harris and Taylor in \cite{ht}.
\newline




Let us also choose a $p$-adic field $E$ with the ring of integers $\mc{O}$ and the residue field $k$.
\newline
\newline
We will consider the Shimura varieties $S_K$ for $G$ which arises from the moduli problem $M_K$ described as follows: $M_K$ is the functor from the category of pairs $(S, s)$, where $S$ is a connected locally Noetherian $F$-scheme and $s$ is a geometric point of $S$, to the category of sets, defined by sending a pair $(S, s)$ to the set of isogeny classes of quadruples $(A, \lambda, i, \bar{\alpha} )$, where
\newline
\newline
(1) $A$ is an abelian scheme over $S$.
\newline
(2) $\lambda : A \rightarrow A^{\vee}$ is a polarization.
\newline
(3) $i : D \hookrightarrow End_S(A) \otimes _{\mb{Z}} \mb{Q}$ such that $\lambda \circ i(f) = i(f^*) ^{\vee} \circ \lambda$ for all $f \in D$
\newline
(4) $\bar{\alpha}$ is a $\pi _1(S,s)$-invariant $K$-orbit of isomorphisms of $F \otimes _{\mb{Q}} \mb{A} _f$-modules $\alpha : V \otimes _{\mb{Q}} \mb{A} _f \simeq VA_s$, which take the pairing $\left<\cdot,\cdot \right>$ on $V = F^n$ to a $\mb{A} _f ^{\times}$-multiple of the $\lambda$-Weil pairing on $VA_s = H_1(A_s, \mb{A} _f)$. For more details, see section 5 of \cite{ko2}.  
\newline
(5) Kottwitz' determinant condition holds, i.e. for each $f\in F$, there is an equality of polynomials $\det _{\mc{O} _S} (f|Lie A)= det _E(f | V^1)$ (here $V^1$ is a certain subspace of $V \otimes _{\mb{Q}} E$). For details, see section 5 of \cite{ko2} or section 5 of \cite{sh}.
\newline
(6) Two such quadruples $(A, \lambda, i, \bar{\alpha} )$ and $(A', \lambda ', i', \bar{\alpha}' )$ are isogenous if there exists an isogeny $A \ra A'$ taking $\lambda, i, \bar{\alpha}$ to $\gamma \lambda, i', \bar{\alpha}' $ for some $\gamma \in \mb{Q} ^{\times}$.
\newline
\newline
The moduli problem $M_K$ is a smooth separated algebraic stack which is representable by a quasi-projective scheme if the objects it parameterizes have no nontrivial automorphism, so in particular when $K$ is sufficiently small, for example when $K$ is neat (for the definition, see below).

Denote by $X$ the set of $G(\mathbb{R})$-conjugates of $h$. Let $S_K$ be the canonical model over $F$ of the Shimura variety whose $\mathbb{C}$ points are defined by:
$$S_H(\mathbb{C}) = G(\mathbb{Q}) \backslash (G(\mathbb{A} _{f}) \times X) / H$$
we have a bijection $S_H(\mathbb{C}) \simeq M_H(\mathbb{C})$ of underlying sets.

Recall that to an algebraic finite-dimensional representation $W$ over $E$ of $G$, we can associate a local system $\mathcal{V}(W)$ on $S_K$ by the construction described in chapter 4 of \cite{ht} (see also chapter 3 of \cite{mi}). Here we use the fact that the Galois group of $S_{K'}$ over $S_{K}$ for $K' \subset K$ is equal to $K/ K'$ (see the beginning of chapter 4 of \cite{ht}).

If $K^p$ is some fixed compact open subgroup of $G(\mathbb{A} _f ^p)$, then we write:

$$H^{i} (K^p) _A = \varinjlim _{K_p}  H^i _{\acute{e}t}((S_{K_p K^p}) _{/\bar{F}}, A)$$ 

where the inductive limit is taken over all the compact open subgroups $K_p$ of $G(\mathbb{Q} _p)$ and where $A$ denotes one of $E, \mathcal{O}, \mathcal{O}/\varpi ^s \mathcal{O}$. Write also

$$\widehat{H} ^{i} (K^p) _{\mathcal{O}} = \varprojlim _{s} H^{i} (K^p) _{\mathcal{O}} / \varpi ^s H^{i} (K^p) _{\mathcal{O}} $$

To a finite-dimensional representation $W$ of $G$ over $E$ associate an automorphic vector bundle $\mc{V} _W$ on $S_K$ and define a cohomology group by
$$H^{i}(\mathcal{V} _W) _E = \varinjlim _{K} H^i _{\acute{e}t}((S_{K}) _{/\bar{F}}, (\mathcal{V} _W)_E)$$
where $(\mathcal{V} _W)_E$ denotes the sheaf $\mathcal{V} _W$ with coefficients extended to $E$.
\newline
\newline
Consider for the moment the general situation when $G_w$ is some unramified group over $F_w$ where $w$ is some place of $F$. Choose $K_w$ a hyperspecial subgroup of $G_w$, define the Hecke algebra $\mc{H} _w (G_w)$ as the set of compactly supported $K_w$-biinvariant $\mc{O}$-valued functions on $G_w$. The structure of algebra comes from the convolution. Normalise the Haar measure on $G_w$ so that $K_w$ has volume 1. For an unramified representation $\pi$ of $G_w$, define $\chi _{\pi}: \mc{H} _w(G_w) \ra \mc{O}$ by  $f \mapsto \textrm{tr} \pi (f)$. It is known that $\pi \mapsto \chi _{\pi}$ gives a bijection between unramified representations and the characters of Hecke algebra (see 1.1 of \cite{sh}, also for the references to the proof).

For an algebraic group $G$ over a number field $F$, let $\Sigma$ be a finite set containing all the primes of $F$ at which $G$ is ramified. We put
$$\mc{H} _{\Sigma} (G) = \otimes _{w \not \in \Sigma} \mc{H} _w (G(F_w))$$

Enlarge $\Sigma$ to contain all the places where $K^p$ is not hyperspecial. For a compact open subgroup $K^p$ of $G(\mathbb{A} _f ^p)$ and a compact open sugroup $K_p$ of $G(\mathbb{Q} _p)$ which is normal in $G(\mb{Z}_p)$, let $\mb{T} (K_p K^p) _{\mc{O}}$ denote the image of $\mc{H} _{\Sigma} (G)$ in $\textrm{End} _{\mc{O}[G(\mb{Z}_p)/K_p]} (R\Gamma(K_pK^p, \mc{O}))$ where $R\Gamma(K_pK^p, \mc{O})$ is the cohomology complex of $S_{K_pK^p}$ with coefficients in $\mc{O}$ and endomorphisms are considered in the derived category of $\mc{O}[G(\mb{Z}_p)/K_p]$-modules. Observe that this algebra acts by functoriality on all the cohomology groups.
We will omit often subscript $\mc{O}$ from the notation. If $K_p ' \subset K_p$ is an inclusion then there is a natural surjection $\mathbb{T}  (K_p ' K^p) \rightarrow \mathbb{T}   (K_p K^p)$ which comes from the Hochschild-Serre spectral sequence  $R\Gamma(K_p/K'_p, R\Gamma(K_p' K^p, \mc{O})) \simeq R\Gamma(K_p K^p, \mc{O})$, where we have written $R\Gamma (K_p /K_p',-)$ for the derived complex of the functor $I \mapsto I^{K_p/K_p'}$. Define $\mathbb{T}   (K^p) = \varprojlim _{K_p} \mathbb{T}  (K_pK^p)$ and equip it with its projective limit topology, each of the $\mathcal{O}$-algebras $\mathbb{T}  (K_pK^p)$ being equipped with its $\varpi$-adic topology.




We also define the localisation of the completed cohomology groups at the maximal ideals of the Hecke algebra. For a maximal ideal $\frak{m}$ of $\mc{H} _{\Sigma}  (G)$ and $A = \mathcal{O}$ or $E$, write:
$$\widehat{H} ^{i} (K^p) _{A, \frak{m}} = \mathbb{T}  (K^p) _{\frak{m}} \otimes _{\mathbb{T}   (K^p)} \widehat{H}^{i} (K^p) _{A}$$
where by $\mathfrak{m}$ above, we mean the image of $\mathfrak{m}$ in $\mathbb{T} (K^p)$ (that is, the inverse limit over $K_p$ of images of $\mathfrak{m}$ in each $\mathbb{T} (K_pK^p)$) and 
$$\widehat{H} ^{i} _{A, \frak{m}} = \varinjlim _{K ^p} \widehat{H} ^{i} (K^p) _{A, \frak{m}}$$

Let us now recall some definitions from \cite{em4}. 

\begin{defi}

Let $V$ be a representation of $G$ over $E$. Here $G$ is the group of $\mb{Q} _p$-points in some connected linear algebraic group $\mb{G}$ over $\mb{Q} _p$. For a finite dimensional, algebraic representation $W$ of $\mb{G}$ over $E$, we will write $V_{W-la}$ for the locally $W$-algebraic vectors of $V$ for the action $G$. We say that a vector $v$ in $V$ is locally $W$-algebraic if there exists an open subgroup $H$ of $G$, a natural number $n$, and an $H$-equivariant homomorphism $W^n \ra V$ whose image contains the vector $v$. It is proved in proposition 4.2.2 of \cite{em4}, that $V_{W-la}$ is a $G$-invariant subspace of $V$.

We say that a vector $v$ in $V$ is locally algebraic, if it is locally $W$-algebraic for some finite dimensional algebraic representation $W$ of $\mb{G}$. The set of all locally algebraic vectors of $V$ is a $G$-invariant subspace of $V$,  which we denote by $V _{la}$ (see proposition 4.2.6 of \cite{em4}).

At the beginning of section 3, we will also use the notion of the locally analytic vectors to state lemma 3.2. For the definition, the reader should consult definition 3.5.3 of \cite{em4}. We will denote by $V_{lan}$ the set of locally analytic vectors.

\end{defi}

Let $A$ be a ring, and let $\Gamma$ be a profinite group. Then we define the completed group ring
$$A[[\Gamma]] = \varprojlim _{H} A[\Gamma /H]$$
where $H$ runs over the open subgroups of $\Gamma$.

\section{Cohomology}

We start with the following definition
\begin{defi}
A maximal ideal $\mathfrak{m}$ of $\mathbb{T} (K) $ is cohomologically Eisenstein if $H^i((S_K) _{\bar{F}}, \mathbb{F} _p) _{\mathfrak{m}}$ is non-zero for some $i\not = n-1$. For a compact open subgroup $K^p$ of $G(\mathbb{A} _f ^p)$, A maximal ideal $\mf{m}$ of $\mb{T}(K^p)$ is cohomologically Eisenstein if $H^i((S_{K_pK^p}) _{\bar{F}}, \mathbb{F} _p) _{\mathfrak{m}}$ is non-zero for some $i\not = n-1$ and some compact open subgroup $K_p$ of $G(\mb{Q}_p)$.
\end{defi}

Let us remark, that by the long exact sequence associated to $0 \ra  \mb{Z}_p \ra \mb{Z}_p \ra \mb{F} _p \ra 0$, the maximal ideal $\mathfrak{m}$ of $\mathbb{T} (K) $ is cohomologically Eisenstein if $H^i((S_K) _{\bar{F}}, \mathbb{Z} _p) _{\mathfrak{m}}$ is non-zero for some $i\not = n-1$ or $H^{n-1}((S_K) _{\bar{F}}, \mathbb{Z} _p) _{\mathfrak{m}}$ is not torsion-free.


Let $\frak{m}$ be a cohomologically non-Eisenstein maximal ideal of some $\mathbb{T}(K)$.

\begin{lemm} 
There is a natural isomorphism
$$H^{n-1}(\mathcal{V} _W)_{E, \mathfrak{m}} \simeq Hom _{\mathfrak{g}}(   W^{\vee}, (\widehat{H} ^{n-1} _{E, \mathfrak{m}}) _{lan})$$
Here $(\widehat{H} ^{n-1} _{E, \mathfrak{m}}) _{lan}$ denotes the set of locally analytic vectors and $\mf{g}$ is the Lie algebra of $G(\mathbb{Q} _p)$
\end{lemm}

\begin{proof}
This results from the Emerton spectral sequence (2.2.18 of \cite{em3}; see also (2.4) of \cite{em3} where it is explained how the result carries over for general Shimura varieties):
$$ \textrm{Ext} ^i _{\mathfrak{g}} (W^{\vee}, (\widehat{H} ^{j} _E) _{lan}) \Rightarrow H^{i+j}(\mathcal{V} _W)$$
after we localise it at the cohomologically non-Eisenstein ideal $\mathfrak{m}$, because $\widehat{H} ^{j} _{E, \mathfrak{m}} = 0$ for $j \not = n-1$.
\end{proof}


Let $W _i$, ($i\in I$) be a complete set of isomorphism class representatives of the irreducible algebraic representations of $G$ over $E$. Let $B_i = End _G(W_i)$. We obtain


\begin{prop} 

The evaluation map gives a $G_{F} \times G(\mathbb{A} _f)$-equivariant isomorphism:
$$\bigoplus _{i \in I, n\in \mathbb{Z}} H^{n-1}(\mathcal{V} _{W_i})_{E,  \mathfrak{m}} \otimes _{B_i} W_{i}^{\vee}  \simeq (\widehat{H} ^{n-1} _{E, \mathfrak{m}} ) _{la}$$

Here $G_{F} \times G(\mathbb{A} _f)$ acts on $W_i ^{\vee}$ through its quotient $G(\mathbb{Q} _p)$.
\end{prop}

\begin{proof}

Let us put $V = \widehat{H} ^{n-1} _{E, \frak{m}}$. By the corollary 4.2.7 in \cite{em4}, we have:
$$\bigoplus _{W} V_{W-la} \simeq V_{la}$$
where the sum runs over the complete set of isomorphism class representatives of the irreducible algebraic representations $W$ of $G$. By Proposition 4.2.10 of \cite{em4} we have a topological isomorphism
$$Hom _{\mf{g}}(W, V_{lan}) \otimes _{B} W \simeq V_{W-la}$$
where $B = End _G(W)$. We use these facts and the lemma 3.2 to conclude.

\end{proof}

We say that a cuspidal automorphic representation $\pi = \pi _{\infty} \otimes \pi _f$ of $G(\mathbb{A})$ occurs in $H^{i}(S_K, \mathcal{V} _W)$ if the $\pi _f$-isotypical component $W_{\pi}  ^i = Hom _{G(\mathbb{A} _f)}(\pi _f, H^{i}(\mathcal{V} _W))$ is non-zero and $\pi _{\infty}$ is cohomological for the representation $W$. We have a decomposition:
\begin{enumerate}
\item[(A)]
\hfill $H^{i}(S_K, \mathcal{V} _W) = \bigoplus _{\pi} \pi _f ^K \otimes W _{\pi} ^i$ \hfill~
\end{enumerate}
where $\pi$ runs over automorphic representations of weight $W$.

\section{Density result}

Fix a finite set $\Sigma _0$ of rational primes in $\mb{Q}$, not containing $p$ and containing all the rational primes which divide the primes in $F$ at which $D$ is ramified. We will also denote by $\Sigma _0$ the set of primes in $F$ (or in $F^+$) which lie over those in $\Sigma _0$. It should not cause any confusion. Let $\Sigma = \Sigma _0 \cup \{p \}$.

Let $K_0 ^{\Sigma } = \prod _{l \notin \Sigma} G(\mathbb{Z} _{l})$ and let $G = G(\mathbb{Q} _p)$. Define also $G _{\Sigma  _0} = \prod _{l \in \Sigma _0} G(\mathbb{Q} _l)$ and $G_{\Sigma} = G(\mb{Z}_p)G_{\Sigma_0}$. We fix a maximal compact subgroup $K_{\Sigma, 0} = \prod _{l \in \Sigma} G(\mb{Z}_l)$ and in the rest of this article we will consider only those compact open subgroups $K_{\Sigma} \subset G_{\Sigma}$ which are normal in $K_{\Sigma, 0}$. For such a compact open sugroup $K_{\Sigma} \subset G_{\Sigma}$, we write $\mathbb{T}  (K_{\Sigma})$ for the image of $\mc{H} _{\Sigma} (G)$ in $End _{\mc{O}[K_{\Sigma, 0}/K_{\Sigma}]}(R\Gamma(K_{\Sigma}K^{\Sigma} _0))$ (endomorphisms are considered in the derived category of $\mc{O}[K_{\Sigma, 0}/K_{\Sigma}]$-modules). This gives us a compatible action of Hecke algebra on our tour of Shimura varieties. For any such $K_{\Sigma} ' \subset K_{\Sigma}$ we have a surjection $\mb{T}(K_{\Sigma} ') \twoheadrightarrow \mb{T}(K_{\Sigma})$. We define $\mb{T}(K_{\Sigma _0}) = \varprojlim _{K_p} \mb{T}(K_{\Sigma _0} K_p)$. Finally we put $\mb{T} _{\Sigma} = \varprojlim _{K_{\Sigma}} \mb{T}(K_{\Sigma})$. 

We fix a cohomologically non-Eisenstein ideal $\mf{m}$ of $\mb{T} _{\Sigma} = \varprojlim _{K_{\Sigma _0}} \mathbb{T}  (K _{\Sigma _0})$, where cohomologically non-Eisenstein means the vanishing of the cohomology groups of  $H^{i}(S_{K_{\Sigma _0} K_p K^{\Sigma} _0}, \mb{F}_p) _{\mf{m}}$ for $i \not = p$ and for all $K_{\Sigma _0}$ and $K_p$.

\begin{defi}
We call a compact open subgroup $K_{\Sigma _0}$ an allowable level for $\mf{m}$, if the image of $\mf{m}$ in $\mathbb{T}  (K _{\Sigma _0})$ is a proper maximal ideal.
\end{defi}


For $A = E, \mathcal{O}, \mathcal{O}/\varpi ^s \mathcal{O}$, we will write 
$$H^{i}(K_{\Sigma _0}) _A = H^{i}(K_{\Sigma _0} K_0 ^{\Sigma})_A$$
and similarly for $\widehat{H}^{i} (K _{\Sigma _0}) _{A}$.
We also put
$$\widehat{H} ^{i} (K _{\Sigma _0}) _{\mc{O}, \mf{m}} = \mathbb{T}  (K _{\Sigma _0}) _{\mf{m}} \otimes _{\mathbb{T}  (K _{\Sigma _0})} \widehat{H}^{i} (K _{\Sigma _0}) _{\mb{O}}$$
and
$$\widehat{H} ^{i} _{\mc{O}, \mf{m}, \Sigma} = \varinjlim _{K_{\Sigma _0}} \widehat{H} ^{i}(K_{\Sigma _0}) _{\mc{O}, \mf{m}}$$ 
where the limit runs over all the allowable levels $K_{\Sigma _0}$ for $\mf{m}$. By $\widehat{H} ^{i} (K _{\Sigma _0}) _{E, \mf{m}}$ we will mean $\widehat{H} ^{i} (K _{\Sigma _0}) _{\mc{O}, \mf{m}} \otimes _{\mc{O}} E$ and similarly for $\widehat{H} ^{i} (K _{\Sigma _0}) _{E, \mf{m}}$

\begin{rema}

We will use the notion of neatness for compact open subgroups $K_f$ of $G(\mathbb{A} _f)$. For that, see section 0.6 in \cite{pi} for a precise definition. We will only need this condition to ensure that $K_f$ acts on $G(\mb{Q}) \backslash (X \times G(\mathbb{A} _f))$ without fixed points, so that we can use Hochschild-Serre spectral sequence. Actually, any sufficiently small open compact subgroup is neat (see also 0.6 in \cite{pi})

\end{rema}

Let us first consider an auxilary lemma:

\begin{lemm}

Let $\Gamma$ be a finite group and let $V$ be a finitely generated representation of $\Gamma$ over $\mathcal{O}/\varpi ^s \mathcal{O}$. Then if $V$ is of finite injective dimension as a representation of $\Gamma$, then $V$ is injective.

\end{lemm}

\begin{proof}

Dualizing the situation, by assumption we know that $V'$ is a finitely generated module over $(\mathcal{O}/\varpi ^s \mathcal{O})[\Gamma]$ which is of finite projective dimension. Take a finite, projective resolution of $V'$
$$ 0 \ra P_n  \ra ... \ra P_1  \ra V' \ra 0$$ 
so that $P_i$ are finitely generated projective $(\mathcal{O}/\varpi ^s \mathcal{O})[\Gamma]$-modules (because $(\mathcal{O}/\varpi ^s \mathcal{O})[\Gamma]$ is self-dual, i.e. $(\mathcal{O}/\varpi ^s \mathcal{O})[\Gamma]' = (\mathcal{O}/\varpi ^s \mathcal{O})[\Gamma]$), and hence each $P_i$ is a direct factor of $(\mathcal{O}/\varpi ^s \mathcal{O})[\Gamma]^{r_i}$ for some $r_i >0$. Dualizing again, we obtain an injective resolution of $V$ of the form
$$0 \ra V \ra P_1 ' \ra ... \ra P_{n-1} ' \ra P_n ' \ra 0$$ 
But each $P_i '$ is again a direct factor of $(\mathcal{O}/\varpi ^s \mathcal{O})[\Gamma]^{r_i}$, as $(\mathcal{O}/\varpi ^s \mathcal{O})[\Gamma]' = (\mathcal{O}/\varpi ^s \mathcal{O})[\Gamma]$, and hence each $P_i '$ is again projective. This means that the last surjection  $P_{n-1} ' \ra P_n '$ splits and hence we can shorten choosen injective resolution. Continuing this process, we arrive at the isomorphism $V \simeq P'$, where $P'$ is some injective module.

\end{proof}

Consider a compact open subgroup $K_p$ of $G$ and fix $K^p \subset G(\mathbb{A} ^p _F)$ as in the lemma below. Recall that the Galois group of $\varprojlim _{K_p ' \subset K_p} S_{K_p 'K^p}$ over $S_{K_pK^p}$ is equal to $K_p$ in the case of Shimura varieties of Harris-Taylor, but for more general Shimura varieties it might be a proper quotient of $K_p$ and hence we will note it $L_p$ for the sequel in order to show that nothing changes in general for the following lemma. We will define in the same manner $L_p'$ for other choice of $K_p ' \subset G$.
\begin{lemm}
If $K_p$ is a compact open subgroup of $G =G(\mathbb{Q} _p)$, and if $K_{\Sigma _0} \subset G_{\Sigma _0}$ is an allowable level, chosen so that $K_p K_{\Sigma _0} K_0 ^{\Sigma _0}$ is neat, then

a) for each $s>0$, $H^{n-1}(K_{\Sigma _0})_{\mathcal{O}/\varpi ^s \mathcal{O}, \mf{m}}$ is injective as a smooth representation of $L_p$ over $\mathcal{O} /\varpi ^s \mathcal{O}$.

b) $H^{i}(M(K_p K_{\Sigma _0} K_0 ^{\Sigma}), \mathcal{W} ^{\vee}) _{\mf{m}} =0$ for all $i \not = n-1$, where $\mathcal{W}$ is a local system induced by a finitely generated smooth representation $W$ of $L_p$ over $\mathcal{O} /\varpi ^s \mathcal{O}$. 
\end{lemm}

\begin{proof}
a) Take a finitely generated smooth representation $W$ of $L_p$ over $\mathcal{O} /\varpi ^s \mathcal{O}$, consider its smooth Pontrjagin dual and the constant local system $\mathcal{W} ^{\vee}$ it induces on our Shimura variety $S_{K_p K_{\Sigma _0} K_0 ^{\Sigma _0}}$ by the well-known correspondence between local systems and representations of the Galois group.
For $K_p '$ sufficiently small (i.e. for which $L_p '$ acts trivially on $W$ and hence $\mathcal{W}$ is a constant sheaf on $M(K_p ' K_{\Sigma _0} K_0 ^{\Sigma})$) we have 
$$H^{r}(M(K_p ' K_{\Sigma _0} K_0 ^{\Sigma}), \mathcal{W} ^{\vee}) \simeq H^{r}(M(K_p ' K_{\Sigma _0} K_0 ^{\Sigma}), \mathcal{O} /\varpi ^s \mathcal{O}) \otimes W ^{\vee} \simeq Hom  (W, H^{r}(M(K_p ' K_{\Sigma _0} K_0 ^{\Sigma}), \mathcal{O} /\varpi ^s \mathcal{O}))$$
Consider Hochschild-Serre spectral sequence:
$$H^i(L_p / L_p ', H^j(M(K_p' K_{\Sigma _0} K_0 ^{\Sigma}), \mathcal{W} ^{\vee})) \Rightarrow H^{i+j}(M(K_p K_{\Sigma _0} K_0 ^{\Sigma}), \mathcal{W} ^{\vee})$$
Taking for the moment $W$ to be a trivial representation and localising the spectral sequence at $\mf{m}$ which is cohomologically non-Eisenstein, we get by looking at the $i+j=n-1$ an isomorphism
$$H^{n-1}(M(K_p' K_{\Sigma _0} K_0 ^{\Sigma}), \mathcal{O}/\varpi ^s \mathcal{O}) _{\mf{m}} ^{L_p} \simeq H^{n-1}(M(K_p K_{\Sigma _0} K_0 ^{\Sigma}), \mathcal{O}/\varpi ^s \mathcal{O}) _{\mf{m}}$$
Take direct limit over $K_p '$ to obtain
$$H^{n-1}(K_{\Sigma _0}) ^{L_p} _{\mathcal{O}/\varpi ^s \mathcal{O}, \mf{m}} \simeq H^{n-1}(M(K_p K_{\Sigma _0} K_0 ^{\Sigma}), \mathcal{O}/\varpi ^s \mathcal{O}) _{\mf{m}}$$
Using this isomorphism for $K_p '$ and the isomorphism mentioned before we have
$$H^i(L_p / L_p ', H^j(M(K_p' K_{\Sigma _0} K_0 ^{\Sigma}), \mathcal{W} ^{\vee})) \simeq H^i(L_p / L_p ' , Hom  (W, H^{j}(K_{\Sigma _0}) ^{L_p '} _{\mathcal{O}/\varpi ^s \mathcal{O}, \mf{m}})) \simeq \textrm{Ext} ^i _{L_p / L_p '} (W, H^{j}(K_{\Sigma _0}) ^{L_p '} _{\mathcal{O}/\varpi ^s \mathcal{O}, \mf{m}})$$
and so our spectral sequence transforms to
$$\textrm{Ext} ^i _{L_p / L_p '} (W, H^{j}(K_{\Sigma _0}) ^{L_p '} _{\mathcal{O}/\varpi ^s \mathcal{O}, \mf{m}}) \Rightarrow H^{i+j}(M(K_p K_{\Sigma _0} K_0 ^{\Sigma}), \mathcal{W} ^{\vee}) _{\mf{m}}$$
and hence
$$\textrm{Ext} ^i _{L_p / L_p '} (W, H^{n-1}(K_{\Sigma _0}) ^{L_p '} _{\mathcal{O}/\varpi ^s \mathcal{O}, \mf{m}}) \simeq H^{n-1+i}(M(K_p K_{\Sigma _0} K_0 ^{\Sigma}), \mathcal{W} ^{\vee}) _{\mf{m}}$$
Observe that, because $H^{i+j}(M(K_p K_{\Sigma _0} K_0 ^{\Sigma}), \mathcal{W} ^{\vee}) _{\mf{m}} = 0$ for $i+j > 2n-2$, we have
$$\textrm{Ext} ^i _{L_p / L_p '} (W, H^{j}(K_{\Sigma _0}) ^{L_p '} _{\mathcal{O}/\varpi ^s \mathcal{O}, \mf{m}}) = 0$$
for $i > n-1$ and hence $H^{n-1}(K_{\Sigma _0})_{\mathcal{O}/\varpi ^s \mathcal{O}, \mf{m}} ^{L_p '}$ is of finite injective dimension as a representation of $L_p / L_p '$. The group $L_p / L_p '$ is finite and hence we can use the lemma proved above. Applying it to $\Gamma = L_p/ L_p'$ and $V = H^{n-1}(K_{\Sigma _0})_{\mathcal{O}/\varpi ^s \mathcal{O}, \mf{m}} ^{L_p'}$ shows that in fact $H^{n-1}(K_{\Sigma _0})_{\mathcal{O}/\varpi ^s \mathcal{O}, \mf{m}} ^{L_p'}$ is injective as a representation of $L_p / L_p'$ and so $\textrm{Ext} ^i _{L_p / L_p '}(W, H^{n-1}(K_{\Sigma _0})_{\mathcal{O}/\varpi ^s \mathcal{O}, \mf{m}} ^{L_p'}) = 0$ for $i \not = 0$

Consider Hochschild-Serre spectral sequence for a pro-scheme $\varprojlim _{K_p} M(K_pK_{\Sigma _0} K_0 ^{\Sigma})$ over a scheme $M(K_p ' K_{\Sigma _0} K_0 ^{\Sigma})$ after localisation at $\mf{m}$
$$H^i(L_p ', H^{j}(K_{\Sigma _0})_{\mathcal{O}/\varpi ^s \mathcal{O}, \mf{m}}) \Rightarrow H^{i+j}(M(K_p ' K_{\Sigma _0} K_0 ^{\Sigma}), \mathcal{O}/\varpi ^s \mathcal{O} ) _{\mf{m}}$$
 By assumption that $\mathfrak{m}$ is cohomologically non-Eisenstein, we know that $H^{i+j}(M(K_p ' K_{\Sigma _0} K_0 ^{\Sigma}), \mathcal{O}/\varpi ^s \mathcal{O} ) _{\mf{m}} = 0$ for $i+j \not = n-1$ and so we conclude that $H^i(L_p ', H^{n-1}(K_{\Sigma _0})_{\mathcal{O}/\varpi ^s \mathcal{O}, \mf{m}}) = 0$ for $i \not = 0$. 

Because the functor which takes $I$ to $I^{L_p '}$, where $I$ is $L_p$-representation, maps injective objects to injective objects, we can derive the functor $(Hom _{L_p}(W, - ) ^{L_p '})^{L_p / L_p'}$ to obtain a spectral sequence
$$\textrm{Ext} ^i _{L_p / L_p '}(W, H^j(L_p ', H^{n-1}(K_{\Sigma _0})_{\mathcal{O}/\varpi ^s \mathcal{O}, \mf{m}})) \Rightarrow \textrm{Ext} ^{i+j} _{L_p }(W, H^{n-1}(K_{\Sigma _0})_{\mathcal{O}/\varpi ^s \mathcal{O}, \mf{m}})$$
As $H^j(L_p ', H^{n-1}(K_{\Sigma _0})_{\mathcal{O}/\varpi ^s \mathcal{O}, \mf{m}}) = 0$ for $j \not = 0$, the spectral sequence degenerates to:
$$\textrm{Ext} ^i _{L_p / L_p '}(W, H^{n-1}(K_{\Sigma _0})_{\mathcal{O}/\varpi ^s \mathcal{O}, \mf{m}} ^{L_p'}) \simeq \textrm{Ext} ^{i} _{L_p }(W, H^{n-1}(K_{\Sigma _0})_{\mathcal{O}/\varpi ^s \mathcal{O}, \mf{m}})$$
We have proved above that $H^{n-1}(K_{\Sigma _0})_{\mathcal{O}/\varpi ^s \mathcal{O}, \mf{m}} ^{L_p'}$ is injective as a representation of $L_p / L_p'$ and so $\textrm{Ext} ^i _{L_p / L_p '}(W, H^{n-1}(K_{\Sigma _0})_{\mathcal{O}/\varpi ^s \mathcal{O}, \mf{m}} ^{L_p'}) = 0$ for $i \not = 0$, which implies by the above isomorphism that $\textrm{Ext} ^{i} _{L_p }(W, H^{n-1}(K_{\Sigma _0})_{\mathcal{O}/\varpi ^s \mathcal{O}, \mf{m}}) = 0$ for $i \not = 0$, i.e. $H^{n-1}(K_{\Sigma _0})_{\mathcal{O}/\varpi ^s \mathcal{O}, \mf{m}}$ is an injective representation of $L_p$ over $\mathcal{O} /\varpi ^s \mathcal{O}$ as $W$ is an arbitrary finitely generated smooth representation.

b) After a) $H^{n-1}(K_{\Sigma _0})_{\mathcal{O}/\varpi ^s \mathcal{O}, \mf{m}}$ is an injective representation of $L_p$ over $\mathcal{O} /\varpi ^s \mathcal{O}$. Let us look again at the spectral sequence
$$\textrm{Ext} ^i _{L_p} (W, H^{j}(K_{\Sigma _0})_{\mathcal{O}/\varpi ^s \mathcal{O}, \mf{m}}) \Rightarrow H^{i+j}(M(K_p K_{\Sigma _0} K_0 ^{\Sigma}), \mathcal{W} ^{\vee}) _{\mf{m}}$$
to see that, as $\textrm{Ext} ^i _{L_p} (W, H^{n-1}(K_{\Sigma _0})_{\mathcal{O}/\varpi ^s \mathcal{O}, \mf{m}}) = 0$ for $i > 0$, we have $H^{i+j}(M(K_p K_{\Sigma _0} K_0 ^{\Sigma}), \mathcal{W} ^{\vee}) _{\mf{m}} =0$ for all $i+j \not =n-1$.

\end{proof}

\begin{rema}

Observe that point b) tells that assuming that our ideal $\mathfrak{m}$ is cohomologically non-Eisenstein, we obtain a vanishing result for $H^{\bullet}(M(K_p K_{\Sigma _0} K_0 ^{\Sigma}), \mathcal{W} ^{\vee}) _{\mf{m}}$ for all local systems coming from finitely generated smooth representations $W$ of $L_p$ over $\mathcal{O} /\varpi ^s \mathcal{O}$. For similar results in the literature, see \cite{mt} and \cite{di} where the authors obtain vanishing results for different Shimura varieties also after localising at a Hecke ideal. In a slightly different vein, a general vanishing result for automorphic sheaves on Shimura varieties of PEL-type is formulated and proved in \cite{ls}. The vanishing result is proved without localising at a Hecke ideal, but with additional conditions on considered sheaves. 

\end{rema}

\begin{prop}
Let $K_p$ be a compact open subgroup of $G(\mb{Q}_p)$. Then $\left( \widehat{H} ^{n-1} _{E , \mf{m}, \Sigma} \right) ^{K_p}$ is an $E[G_{\Sigma _0}]$-module of finite length. 
\end{prop}
\begin{proof}
By Proposition 3.3 (and the discussion in the section 4 of \cite{he}) we have
$$\left( \widehat{H} ^{n-1} _{E , \mf{m}, \Sigma} \right) ^{K_p} = \bigoplus _{\pi} \pi _p ^{K_p} \otimes \pi _{\Sigma _0}$$
where $\pi$ runs over automorphic representations which are cohomological of trivial weight, are unramified outside $\Sigma$, have nonzero $K_p$ invariants and where we have denoted $\pi _{\Sigma _0} = \otimes _{l \in \Sigma _0} \pi _l$. Moreover to each representation $\pi$ as above one can attach a p-adic Galois representation $\rho _{\pi}$ of $Gal(\bar{F}/F)$ such that for all places $v=ww^c$ of $F^+$ which are split in $F$, $\pi _v \circ i_w$ corresponds to $\rho _{\pi, w}$ by the Local Langlands correspondence ($i_w$ is an isomorphism between $G(F ^+ _v)$ and $GL_n(F_w)$) and each $\rho _{\pi}$  is a lift of a mod p Galois representation $\bar{\rho} _{\mf{m}}$ associated to $\mf{m}$ (see the section 5 for the precise definition of $\bar{\rho} _{\mf{m}}$, which we do not need here). There are only finitely many automorphic lifts of $\bar{\rho} _{\mf{m}}$ of given weight, the set of possible ramification and which have non-zero $K_p$-invariants. Indeed, all such automorphic lifts will be of bounded conductor: non-zero $K_p$-invariants force a bound on $p$-conductor and the outside-$p$ part of the conductor is bounded by the result of Livne (see proposition 1.1 in \cite{li} where is proved an equality of Swan conductors of a lift and the reduction; this implies the result as the conductor is the sum of tame part and Swan conductor and the tame part has the obvious bound). Thus, there are only finitely many $\pi$ in the sum above and hence we can conclude as each $\pi _p ^{K_p}$ is finite dimensional and $\left( \widehat{H} ^{n-1} _{E , \mf{m}, \Sigma} \right) ^{K_p}$ is admissible.
\end{proof}

\begin{coro}
Let $K_p$ be a compact open subgroup of $G(\mb{Q}_p)$. Then $\left( H ^{n-1} _{k , \mf{m}, \Sigma} \right) ^{K_p}$ is finitely generated as an $k[G_{\Sigma _0}]$-module.
\end{coro}
\begin{proof} 
By the proposition above, $\left( \widehat{H} ^{n-1} _{E , \mf{m}, \Sigma} \right) ^{K_p}$ is of finite length. We know that every $\mc{O}$-stable lattice in a smooth $E$-representation of $GL_n(\mb{Q}_l)$ of finite length is finitely generated (see proposition 3.3 in \cite{vi}). Hence $\left( \widehat{H} ^{n-1} _{\mc{O} , \mf{m}, \Sigma} \right) ^{K_p}$ is finitely generated. Because $\mf{m}$ is cohomologically non-Eisenstein we have 
$$\left( \widehat{H} ^{n-1} _{\mc{O} , \mf{m}, \Sigma} \right) ^{K_p} / \varpi \left( \widehat{H} ^{n-1} _{\mc{O} , \mf{m}, \Sigma} \right) ^{K_p} \simeq \left( H ^{n-1} _{k , \mf{m}, \Sigma} \right) ^{K_p}$$
and so we conclude.
\end{proof}

\begin{coro}
If $K_p$ is a pro-p open subgroup of $G$, and if $K_{\Sigma _0} \subset G_{\Sigma _0}$ is an allowable level, chosen so that $K_pK_{\Sigma _0}K_0^{\Sigma _0}$ is neat, then, for some $r>0$, there is an isomorphism $\widehat{H} ^{n-1} (K_{\Sigma _0}) _{\mathcal{O} , \mf{m}} \simeq C(K_p, \mathcal{O})^r$ of $\varpi$-adically admissible $K_p$-representations over $\mathcal{O}$.
\end{coro}
\begin{proof}
As $K_p$ is a pro-p group, the completed group ring $(\mathcal{O} / \varpi ^s \mathcal{O})[[K_p]]$ is a (non-commutative) local ring, thus any non-zero finitely generated projective $(\mathcal{O} /\varpi ^s \mathcal{O})[[K_p]]$-module is isomorphic to $(\mathcal{O} /\varpi ^s \mathcal{O})[[K_p]]^r$ (theorem of Kaplansky). We dualize it and using lemma 4.4, corollary 4.7 and the isomorphism 
$$\widehat{H} ^{n-1} (K_{\Sigma _0}) _{\mathcal{O} , \mf{m}} / \varpi ^s \widehat{H} ^{n-1} (K_{\Sigma _0}) _{\mathcal{O} , \mf{m}} \simeq H^{n-1}(K_{\Sigma _0})_{\mathcal{O}/\varpi ^s \mathcal{O}, \mf{m}}$$
we conclude  that
$$\widehat{H} ^{n-1} (K_{\Sigma _0}) _{\mathcal{O} , \mf{m}} / \varpi ^s \widehat{H} ^{n-1} (K_{\Sigma _0}) _{\mathcal{O} , \mf{m}} \simeq C(K_p, \mathcal{O}/\varpi ^s \mathcal{O})^{r_s}$$
for each $s>0$ and some $r_s >0$. Observe that if we would prove that $r_s$ is independent of $s$ then we could pass to the projective limit in $s$ and arrive at the conclusion. Denote by $M = \widehat{H} ^{n-1} (K_{\Sigma _0}) _{\mathcal{O} , \mf{m}}$ and let $r = r_1$. Then the dual $M'$ of $M$ is  a $\mathcal{O}[[K_p]]$-module and moreover we have $M' / \varpi ^s M' \simeq ((\mathcal{O}/\varpi ^s \mathcal{O})[[K_p]])^{r_s}$. By applying topological Nakayama's lemma (see section 3 in \cite{bh}) to $M' / \varpi ^s M'$, ideal $\varpi( \mathcal{O}/\varpi ^s \mathcal{O})[[K_p]]$ and the local ring $(\mathcal{O}/\varpi ^s \mathcal{O})[[K_p]]$, we get $r_s = r$ as wanted. 

\end{proof}

\begin{prop}
If $K_{\Sigma _0} \subset G_{\Sigma _0}$ is an allowable level, then the space of $G(\mathbb{Z} _p)$-algebraic vectors $(\widehat{H} ^{n-1} (K_{\Sigma _0}) _{E , \mf{m}})_{G(\mathbb{Z}_p)-alg}$ is dense in $\widehat{H} ^{n-1} (K_{\Sigma _0}) _{E , \mf{m}}$.
\end{prop}

\begin{proof}
We choose $K_p$ sufficiently small and such that it is normal in $G(\mathbb{Z} _p)$.  Recall that the Galois group of $M(K_{\Sigma _0})=\varprojlim _{K _p '} M(K_p ' K_{\Sigma _0} K_0 ^{\Sigma})$ (the limit is taken over compact open subgroups of $G$) over $M(G(\mathbb{Z} _p) K_{\Sigma _0} K_0 ^{\Sigma})$ is equal to $G(\mathbb{Z} _p)$.

By the above corollary, the topological dual $\widehat{H} ^{n-1} (K_{\Sigma _0}) _{E , \mf{m}}'$ is free as a module over $E \otimes _{\mathcal{O}} \mathcal{O} [[K_p]]$ and this implies that $\widehat{H} ^{n-1} (K_{\Sigma _0}) _{E , \mf{m}} '$ is projective as a module over $E \otimes _{\mathcal{O}} \mathcal{O} [[G(\mathbb{Z}_p)]]$. Indeed, it follows from the isomorphism of functors:
$$Hom _{E \otimes _{\mathcal{O}} \mathcal{O} [[G(\mathbb{Z} _p)]]}(\widehat{H} ^{n-1} (K_{\Sigma _0}) _{E , \mf{m}}',-) \simeq Hom _{E \otimes _{\mathcal{O}} \mathcal{O} [[K_p]]}(\widehat{H} ^{n-1} (K_{\Sigma _0}) _{E , \mf{m}}',-) ^{G(\mathbb{Z} _p) /K_p}$$
as the target is exact, because of the freeness of $\widehat{H} ^{n-1} (K_{\Sigma _0}) _{E , \mf{m}}'$ over $E \otimes _{\mathcal{O}} \mathcal{O} [[K_p]]$ and the fact that passing to invariants under a finite group $G(\mathbb{Z} _p)/ K_p$ is exact as $E$ is of characteristic 0.
 
Hence, dualising it again, we find that $\widehat{H} ^{n-1} (K_{\Sigma _0}) _{E , \mf{m}}$ may be $G(\mb{Z}_p)$-equivariantly embedded as a topological direct summand of $C(G(\mb{Z}_p), E)^s$ for some $s>0$. To prove that $(\widehat{H} ^{n-1} (K_{\Sigma _0}) _{E , \mf{m}})_{G(\mb{Z}_p)-alg}$ is dense in $\widehat{H} ^{n-1} (K_{\Sigma _0}) _{E , \mf{m}}$ it suffices to prove that $C(G(\mb{Z}_p), E) _{G(\mb{Z}_p)-alg}$ is dense in $C(G(\mb{Z}_p), E)$.

Because $G(\mathbb{Z} _p) \simeq \mb{Z} _p ^{\times} \times \prod _{v|p} GL_n(\mb{Z} _p)$ is an open, closed subset of $\mathbb{Z} _p ^{dn^2+1}$, we can consider continuous functions on $G(\mb{Z}_p)$ as continuous functions on $\mathbb{Z} _p ^{dn^2+1}$. By the result of Mahler on expansions of continuous p-adic functions, we know that each function on $\mathbb{Z} _p ^{dn^2+1}$ can be written as a power series, so the set of polynomials on $\mathbb{Z} _p ^{dn^2+1}$ is dense. Restricting approximating polynomials to $G(\mb{Z}_p)$ shows that the polynomial functions on $G(\mb{Z}_p)$ are dense in $C(G(\mb{Z}_p),E)$. We conclude by observing that this set of polynomials is contained in $C(G(\mb{Z}_p), E) _{G(\mb{Z}_p)-alg}$.
\end{proof}

Let $\mathbb{T}  _{\Sigma, \mf{m}} = \varprojlim _{K _{\Sigma _0}} \mathbb{T}  (K_{\Sigma _0}) _{\mf{m}}$ where the limit runs over all the allowable levels for $\mf{m}$ and make the following definition:

\begin{defi}

We say that a closed point $\mathfrak{p} \in Spec \mathbb{T}  _{\Sigma, \mf{m}}[1/p]$ (resp. of $Spec \mathbb{T}  (K_{\Sigma _0}) _{\mf{m}} [1/p]$) is an automorphic point if the system of Hecke eigenvalues $Spec \mathbb{T}  _{\Sigma, \mf{m}}[1/p] \rightarrow \kappa (\frak{p})$ (resp. $Spec \mathbb{T}  (K_{\Sigma _0}) _{\mf{m}} [1/p] \ra \kappa (\frak{p})$) determined by $\frak{p}$ arises from an automorphic, cohomological (appearing in the decomposition of $H^{n-1}(\mathcal{V} _W)$ for some representation $W$) representation.

\end{defi}

Remark that $\mf{p}$ is an automorphic point if and only if $\widehat{H} ^{n-1} _{E, \mf{m}, \Sigma} [\frak{p}] _{la}$ is non-zero. This follows from the proposition 3.3.

\begin{defi}
If $K_{\Sigma _0} \subset G_{\Sigma _0}$ is any allowable level for $\mf{m}$, then we let $C(K_{\Sigma _0})$ denote the subset (of crystalline points) of closed points $\frak{p} \in Spec \mathbb{T} (K_{\Sigma _0}) _{\mf{m}} [1/p]$ that are automorphic and whose associated Galois representations (that is $W _{\pi} ^{n-1}$ from the decomposition (A) in section 3 for all $\pi$ which corresponds to $\mathfrak{p}$) are crystalline at each $v|p$, $v \in F$. Let $C$ denote the subset of closed points $\frak{p} \in Spec \mathbb{T}  _{\Sigma, \mf{m}} [1/p]$ that are automorphic and whose associated Galois representations are crystalline at each $v|p$.
\end{defi}

Also here, we can remark that $\mf{p}$ is a crystalline point if and only if $\widehat{H} ^{n-1} _{E, \mf{m}, \Sigma} [\frak{p}] _{la}$ is non-zero and the Galois action on it is crystalline.

\begin{coro}

The direct sum $\bigoplus _{\frak{p} \in C } \widehat{H} ^{n-1} _{E, \mf{m}, \Sigma} [\frak{p}] _{la} $ is dense in $\widehat{H} ^{n-1} _{E, \mf{m}, \Sigma}$.

\end{coro}

Here $\widehat{H} ^{n-1} _{E, \mf{m}, \Sigma} [\frak{p}]$ means the subrepresentation of $\widehat{H} ^{n-1} _{E, \mf{m}, \Sigma}$ on which $\mf{p}$ acts trivially.

\begin{proof}
First of all, observe that it suffices to prove that $\bigoplus _{\frak{p} \in C(K_{\Sigma _0})} \widehat{H} ^{n-1}(K_{\Sigma _0}) _{E, \mf{m}} [\frak{p}] _{la} $ is dense in $\widehat{H} ^{n-1}(K_{\Sigma _0}) _{E, \mf{m}}$. for any allowable level $K_{\Sigma _0} \subset G_{\Sigma _0}$. Proposition above shows that $(\widehat{H} ^{n-1} (K_{\Sigma _0}) _{E , \mf{m}})_{G(\mathbb{Z}_p)-alg}$ is dense in $\widehat{H} ^{n-1} (K_{\Sigma _0}) _{E , \mf{m}}$. Thus $E[G](\widehat{H} ^{n-1} (K_{\Sigma _0}) _{E , \mf{m}})_{G(\mathbb{Z}_p)-alg}$ (the $E[G]$-representation generated by $(\widehat{H} ^{n-1} (K_{\Sigma _0}) _{E , \mf{m}})_{G(\mathbb{Z}_p)-alg}$) is also dense in $\widehat{H} ^{n-1} (K_{\Sigma _0}) _{E , \mf{m}}$. From proposition 3.4 and the formula $H^{n-1}(S_K, \mathcal{V} _W) = \bigoplus _{\pi} \pi _f ^K \otimes W _{\pi} ^{n-1}$, we deduce:

$$ E[G](\widehat{H} ^{n-1} (K_{\Sigma _0}) _{E , \mf{m}})_{G(\mathbb{Z}_p)-alg}  \simeq \left( \bigoplus _{i \in I} \bigoplus _{\ \pi \textrm{ is } W_i-\textrm{coh.}} (\pi _f ^K \otimes W _{\pi} ^{n-1})_{\mathfrak{m}} \otimes _{B_i} W_{i}^{\vee} \right) _{G(\mathbb{Z}_p)-alg} \simeq $$ $$\simeq \bigoplus _{\frak{p}} \widehat{H} ^{n-1} (K_{\Sigma _0}) _{E , \mf{m}} [\mathfrak{p}] _{la}$$

where the sum is taken over all automorphic, cohomological representations which have non-zero $G(\mb{Z} _p)$-algebraic vector, in particular $\pi _v ^{GL_n(\mb{Z} _p)} \not = 0$ for each $v|p$, $v\in F$. In order to see that the sum on the right goes through $C(K _{\Sigma _0})$ (actually, as we don't assume any local-global compatibility, over a possibly smaller subset of $C(K _{\Sigma _0})$), recall that Shimura varieties of PEL-type, when the level at $p$ is hyperspecial, have good reduction at all primes $v$ dividing $p$; this appears in the section 5 of \cite{ko2} and follows from the results of Langlands, Rapoport and Zink. Applying the crystalline conjecture of Fontaine (proved, for example, in \cite{ts}) to each term appearing in the isomorphism from the proposition 3.3, just like in 4.5.4 of \cite{chl}, we conclude that the representations appearing in the cohomology $\widehat{H} ^{n-1} (K_{\Sigma _0}) _{E , \mf{m}} ^{G(\mb{Z} _p)}$ are crystalline at each $v|p$. Here we are also using the fact that the local systems $\mc{V} _{W_i}$ are obtained by tensor operations from the cohomology of an abelian scheme over our Shimura variety (precisely, the cohomology of a Shimura variety with coefficients in $\mc{V} _{W_i}$ equals the image by an idempotent associated to $W_i$ of the cohomology with trivial coefficients of the universal abelian scheme over our Shimura variety) and so we can indeed refer to the classic version of crystalline conjecture which does not involve coefficients. For this, see also 4.5.4 in \cite{chl}.

We conclude that representations appearing in $\widehat{H} ^{n-1} (K_{\Sigma _0}) _{E , \mf{m}} [\mathfrak{p}] _{la} ^{G(\mb{Z} _p)}$ are crystalline at each $v|p$ and those are the Galois representations $W _{\pi} ^{n-1}$. So, the set of points over which we take the sum on the righthand side in the above formula is the subset of $C(K_{\Sigma _0})$. In particular, the result follows.

\end{proof}

\begin{coro}

The set $C$ is Zariski dense in $Spec \mathbb{T}  _{\Sigma, \mf{m}} $, that is, $\bigcap _{\frak{p} \in C} \frak{p} = 0$. 

\end{coro}

\begin{proof}
Suppose that $t \in \bigcap _{\frak{p} \in C} \frak{p}$. Then $t$ annihilates $\bigoplus _{\frak{p} \in C} \widehat{H} ^{n-1} _{E, \mf{m}, \Sigma} [\frak{p}] $ and hence, by the above corollary, $t$ annihilates $\widehat{H} ^{n-1} _{E, \mf{m}, \Sigma} $. But $\mathbb{T}  _{\Sigma, \mf{m}} $ acts faithfully on $\widehat{H} ^{n-1} _{E, \mf{m}, \Sigma}$ and hence $t=0$.
\end{proof}

\section{Eisenstein ideals and Galois representations}

In this section we will discuss the conjectural relation between cohomologically non-Eisenstein ideals and Galois representations. 



The following discussion is based on \cite{he}. Let us denote by $\mf{p}$ a minimal prime ideal of $\mb{T}(K)$. Then $\mf{p}$ determines, for each prime $l$ that splits in $L$, is unramified in $F$ and does not divide the level of $K$, an unramified representation $\pi _{\mf{p}, l}$ of $G(\mb{Q} _l)$ over $\bar{\mb{Q}} _p$ (actually, over $\mb{T}(K) _{\mf{p}}$). Take $W$ to be an irreducible representation of $G$ over $E$. We have
$$H^{n-1}(S_K, \mathcal{V} _W) _{\mf{p}} = \bigoplus _{\pi} H^{n-1}(\mathcal{V} _W)[\pi]^K$$
where the sum is taken over irreducible admissible representations $\pi$ of $G(\mb{A})$ such that
\newline

(i) $\pi ^K \not = 0$

(ii) $\pi _{\infty}$ is cohomological for $W$

(iii) $\pi _{l} \simeq \pi _{\mf{p}, l}$ for all $l$ that split in $L$, are unramified in $F$ and do not divide the level of $K$. For $l$ which does not split in $L$, see lemma 2.2 in \cite{ty}, for the characterisation of $\pi _l$ using base change for unitary groups. 
\newline
\newline
To such a $\pi$ one can associate a Galois representation $\rho _{\pi} : G_F \ra GL_n(\bar{\mb{Q}} _p)$ such that for all primes $v = w w^c$ of $F^+$ which split in $F$,  $\rho _{\pi, w}$ corresponds to $\pi _v \circ i_w$ via the Local Langlands correspondence, where $i_w$ is an isomorphism between $G(F ^+ _v)$ and $GL_n(F_w)$ (see proposition 3.3.4 of \cite{cht}). Moreover, if $\rho _{\pi}$ is irreducible, then $H^{n-1}(\mathcal{V} _W)[\pi]$ is isomorphic (up to semisimplification) to some number of copies of $\pi \otimes _{\bar{\mb{Q}} _p} \rho _{\pi}$ (see proposition 4.1 in \cite{he}). We will denote $\rho _{\pi}$ also by $\rho _{\mf{p}}$ as it depends only on $\mf{p}$, because by Chebotarev density theorem a Galois representation is determined by a dense subset and so, in our case, it is uniquely determined by primes which are unramified in $F$, split in $L$, and do not divide the level $K$. Indeed, one should observe that the set $\{Frob_{w}| w$ is a place of $F$ not in $\Sigma, w \not = w ^c\}$, where $\Sigma$ is some finite set, is dense in $G_{F,\Sigma}$. Denote by $S$ the set of primes of $F^+$ below those of $\Sigma$. By Chebotarev density theorem, we know that $\{Frob_{v}| v$ is a place of $F^+$ not in $S, v = ww ^c$ in $F\}\subset G_{F^+, S}$ has density $\frac{1}{2}$. But we have $G_{F, \Sigma} \subset G_{F^+, S} \twoheadrightarrow Gal(F/F^+)$, thus $G_{F, \Sigma}$ is of index 2 in $G_{F^+, S}$, so that $\{Frob_{w}| w$ is a place of $F$ not in $\Sigma, w \not = w ^c\}$ is dense in $G_{F,\Sigma}$.

Let $\mf{m}$ be the unique maximal ideal containing $\mf{p}$. We have a representation $\bar{\rho} _{\mf{m}} :  G_F \ra GL_n(\mb{T} (K) / \mf{m})$ via reduction modulo $\mf{m}$. That is, choose some automorphic representation $\pi$ corresponding to $\mf{p}$ and let $\rho _{\pi} : G_F \ra GL_n(\bar{\mb{Q}} _p)$ be the Galois representations associated to $\pi$ as above. Choose an invariant lattice in $\rho _{\pi}$, reducing and semisimplyfying gives us the desired representation $\bar{\rho} _{\mf{m}}$. The reader may want to compare this with the proof of proposition 3.4.2 in \cite{cht}. By density again, we see that $\bar{\rho} _{\mf{m}}$ is well-defined up to semisimplification and does not depend on the chosen minimal ideal $\mf{p}$. 

We expect:
\begin{conj}
If the Galois representation $\bar{\rho} _{\mf{m}}$ is absolutely irreducible, then $\mf{m}$ is cohomologically non-Eisenstein.
\end{conj}
For a result in this direction, see appendix A in \cite{he}, where the conjecture is proved under many additional hypotheses on $\bar{\rho} _{\mf{m}}$. Recently, Matthew Emerton and Toby Gee proved a stronger theorem, which is close to the above conjecture for $U(2,1)$ Shimura varieties. Let us cite their theorem B (see also corollary 3.5.1) from \cite{eg}. We refer to their paper for neccessary definitions.
\begin{theo}
Let $X_K$ be a projective $U(2,1)$-Shimura variety of some sufficiently small level $K$. Let $\mf{m}$ be a maximal ideal of the Hecke algebra $\mb{T}(K)$ and let $\bar{\rho} _{\mf{m}} : G_F \ra GL_3(\bar{\mb{F}}_p)$ be the associated Galois representation. Suppose that we have $SL_3(k) \subset \bar{\rho}_{\mf{m}}(G_F) \subset \bar{\mb{F}}_p ^{\times} SL_3(k)$ for some finite extension $k/\mb{F}_p$ and that $\bar{\rho}_{\mf{m} | G_{\mb{Q}_p}}$ is 1-regular and irreducible. Then $\mf{m}$ is cohomologically non-Eisenstein.
\end{theo}

Finally, let us remark, that the reader may want to compare our notion of a cohomologically Eisentein ideal with an Eisenstein ideal of Clozel-Harris-Taylor in \cite{cht} which is defined to be a maximal ideal $\mf{m}$ such that the associated representation $\bar{\rho}_{\mf{m}}$ is absolutely reducible. There is a conjecture B in \cite{cht} related to this notion.


\begin{thebibliography}{9}



\bibitem[BH]{bh} P.N. Balister, S. Howson, "Note on Nakayama's lemma for compact $\Lambda$-modules", Asian Journal of Mathematics 1 (1997) 224-229.

\bibitem[CHL]{chl} L. Clozel, M. Harris, J.-P. Labesse "Construction of automorphic Galois representations I", tome 1 of book project of M. Harris

\bibitem[CHT]{cht} L. Clozel, M. Harris, R. Taylor "Automorphy for some l-adic lifts of automorphic mod l representations", Pub. Math. IHES 108 (2008), 1-181. 

\bibitem[Di]{di} M. Dimitrov, "Galois representations modulo p and cohomology of Hilbert modular varieties", Ann. Sci. École Norm. Sup. 38, Issue 4 (2005), 505-551.

\bibitem[EG]{eg} M. Emerton, T. Gee, "p-adic Hodge-theoretic properties of etale cohomology with mod p coefficients, and the cohomology of Shimura varieties", preprint 2012

\bibitem[Em1]{em1} M. Emerton "Local-global compatibility in the p-adic Langlands programme for $GL_{2 / \mathbb{Q}}$", preprint 2011

\bibitem[Em2]{em2} M. Emerton, "Local-global compatibility conjecture in the p-adic Langlands programme for $GL_{2 / \mathbb{Q}}$", Pure and Applied Math. Quarterly 2 (2006), no. 2, 279-393. 

\bibitem[Em3]{em3} M. Emerton, "On the interpolation of systems of eigenvalues attached to automorphic Hecke eigenforms", Invent. Math. 164 (2006), no. 1, 1-84

\bibitem[Em4]{em4} M. Emerton, "Locally analytic vectors in representations of locally p-adic analytic groups", to appear in Memoirs of the AMS. 

\bibitem[He]{he} D. Helm "Mazur's principle for U(2,1) Shimura varieties", preprint

\bibitem[HT]{ht} M. Harris, R. Taylor "The geometry and cohomology of some simple Shimura varieties", Annals of Math. Studies 151, PUP 2001.

\bibitem[Ko1]{ko1} R. Kottitz, "On the $\lambda$-adic representations associated to some simple Shimura varieties", Invent. Math. 108 (1992), no. 3, 653-665.

\bibitem[Ko2]{ko2} R. Kottwitz , "Points on some Shimura varieties over finite fields", JAMS , 5 (1992) 373-444.


\bibitem[LS]{ls} K.-W. Lan, J. Suh, "Vanishing theorems for torsion automorphic sheaves on compact PEL-type Shimura varieties", preprint 2011

\bibitem[Li]{li} R. Livne, "On the conductor of mod l Galois representations coming from modular forms", J. Number Theory 31 (1989), no. 2, 133-141.

\bibitem[Mi]{mi} J. S. Milne, "Canonical models of (mixed) Shimura varieties and automorphic vector bundles in
automorphic forms, Shimura varieties and L-functions", Automorphic Forms, Shimura Varieties, and L-functions, (Proceedings of a Conference held at the University of Michigan, Ann Arbor, July 6-16, 1988), 1990 

\bibitem[MT]{mt} A. Mokrane, J. Tilouine, "Cohomology of Siegel varieties with p-adic integral coefficients and applications", pp.1-95 (2002), Asterisque 280, Publ. Soc. Math. France

\bibitem[Na]{na} K. Nakamura "Zariski density of crystalline representations for any p-adic field", preprint 2011


\bibitem[Pi]{pi} R. Pink, "Arithmetical compactification of mixed Shimura varieties", Bonner Math. Schriften, 209 (1990).

\bibitem[Sh]{sh} S.W. Shin "Galois representations arising from some compact Shimura varieties", to appear in Annals of Math.

\bibitem[Ts]{ts} T. Tsuji, "p-adic etale cohomology and crystalline cohomology in the semistable reduction case", Invent. Math., 137 (1999), 233-411.

\bibitem[TY]{ty} R. Taylor, T. Yoshida "Compatibility of local and global Langlands correspondences" J.A.M.S. 20 (2007), 467-493. 

\bibitem[Vi]{vi} M.F. Vigneras, "$l$-adic Banach continuous representations of reductive $p$-adic groups when $l \neq p$".  Asterisque 330 1-12 (2010) 

\end{thebibliography}
\end{document}